\documentclass[leqno,twoside, 12pt]{amsart}

\usepackage[square, numbers, comma]{natbib} % Use the natbib reference package - read up on this to edit the reference style; if you want text (e.g. Smith et al., 2012) for the in-text references (instead of numbers), remove 'numbers'. Add sort&compress to sort references the way they appear in the bibliography, and compress references which are close in numbers.
\newtheorem{theorem}{Theorem}[section]

\newtheorem{lemma}[theorem]{Lemma}
\newtheorem{proposition}[theorem]{Proposition}

\newtheorem{definition}{Definition}[section]
\theoremstyle{remark}
\newtheorem{remark}[theorem]{Remark}
\theoremstyle{definition}

\usepackage{setspace}
\numberwithin{equation}{section}
\usepackage[usenames,dvipsnames,svgnames,table]{xcolor}
\usepackage{amsmath,amsthm,amssymb,amssymb,esint,verbatim,tabularx,graphicx,subcaption}
\usepackage{bbm}
\usepackage{fancyhdr}
\usepackage{enumerate}
\usepackage{amsmath}
\usepackage{fancyhdr}
\usepackage{epic}
\usepackage{pgf,tikz}
\usetikzlibrary{arrows}

\usepackage[utf8]{inputenc}
\usepackage{color}
\usepackage{hyperref}
\usepackage{verbatim}
\usepackage{pdfpages}
\usepackage{empheq}

\usepackage[inner=2.5cm,outer=2.5cm,bottom=4cm,top=4cm]{geometry}

\def\RR{\mbox{{\sl I}}\!\mbox{{\sl R}}}

\def\ZZ{\mbox{{\sl Z}}\!\!\mbox{{\sl Z}}}

\newcommand{\pp}{{\textbf{p}}}

\begin{document}
\title[An existence result for the anisotropic $\pp(x)$-Laplacian]{Existence of weak solutions for the anisotropic $\pp(x)$-Laplacian via degree theory}

%\maketitle
%
%{\small
%\begin{center}
%{\sc Pablo Ochoa$^{1}$ and Federico Ramos Valverde$^{2}$} \\
%$^{1}$Corresponding author. Universidad Nacional de Cuyo-CONICET-Universidad Juan A. Maza\\
%Parque Gral. San Mart\'in 5500, Mendoza, Argentina\\ochopablo@gmail.com \\ $^{2}$Universidad Nacional de San Juan-CONICET \\Mitre 396 Este  J5402CWH \\ San Juan,
%Argentina.\\
%federicorvalverde.ffha@gmail.com \
%\end{center}
%}

	\author{Pablo Ochoa}

\address{Pablo Ochoa. Universidad Nacional de Cuyo. CONICET. Universidad J. A. Maza\\Parque Gral. San Mart\'in 5500\\
Mendoza, Argentina.}
\email{pablo.ochoa@ingenieria.uncuyo.edu.ar}
\author{Federico Ramos Valverde}

\address{Federico Ramos Valverde.
Universidad Nacional de San Juan-CONICET \\Mitre 396 Este  J5402CWH \\ San Juan,
Argentina.}
\email{federicorvalverde.ffha@gmail.com}
\author{Anal\'ia Silva}

\address{Silva Anal\'ia. Departamento de Matem\'atica, FCFMyN, Universidad Nacional de San
Luis and Instituto de Matem\'atica Aplicada San
Luis (IMASL), CONICET. Av.Ejercito de los Andes 950, San Luis (5700), Argentina.
}
\email{{acsilva@unsl.edu.ar}\hfill\break\indent {\it Web page:} {\tt analiasilva.weebly.com}}

%\footnote{AMS Subject Classification 2020: 35J62, 46E30, 35D30, 35D40}
%\subjclass {2020}{XX, XX; XX}

%35J62  	Quasilinear elliptic equations
%46E30  	Spaces of measurable functions ($L^p$-spaces, Orlicz spaces, Köthe function spaces, Lorentz spaces, rearrangement invariant spaces, ideal spaces, etc.)
%35D30      Weak solutions to PDEs

	\maketitle
	
	\begin{abstract}
		In this paper, we consider the following Dirichlet boundary value problem involving the anisotropic $\pp(x)$-Laplacian 
\begin{equation}\label{main problemsp}
\begin{cases}
-\Delta_{\pp(x)}u + |u|^{p_M(x)-2}u = f(x, u, D u)\quad \text{in }\Omega\\u= 0 \quad \text{on }\partial \Omega
\end{cases}
\end{equation}where $\pp(x)= (p_1(x), ..., p_n(x))$, with $p_i(x)> 1$ in $\overline{\Omega}$, and $p_M(x):=\max\{p_1(x),..., p_n(x)\}$. Using the topological degree constructed by Berkovits, we prove, under appropriate assumptions on the data, the
existence of weak solutions for the given problem. An important contribution is that we are considering the degenerate and the singular cases in the discussion. Finally,  according to the compact embedding for anisotropic Sobolev spaces, we point out that \eqref{main problemsp} may be critical in some region of $\Omega$. 
\end{abstract}
%\tableofcontents	

	\section{Introduction}
In this paper, we consider weak  solutions of non-homogeneous $\pp(x)$-Laplace equations of the form
\begin{equation}\label{main equation 1}
- \Delta_{\pp(x)} \,u + |u|^{p_M(x)-2}u= f(x,u, Du) \ \ \ \mbox{in} \ \Omega,
\end{equation}where $\Omega$ be a bounded domain of $\RR^n$ ($n \geq 3$) with smooth boundary,  $\pp: \overline{\Omega} \to \RR^n$ is a  vector field 
$$\pp(x)  = (p_1(x), p_2 (x), \ldots, p_n(x)),\quad p_i \in C(\overline{\Omega}), \,p_i(x)> 1 \text{ in }\overline{\Omega}, \,\,i = 1,2, \ldots, n,$$the variable exponent $p_M$ is
$$p_M(x):=\max \{p_1(x),..., p_n(x)\} \quad x \in \Omega,$$
and $- \Delta_{\pp(x)}$ is the $\pp(x)$-Laplace operator defined as
\begin{equation}\label{operator}
     - \Delta_{\pp(x)} \, u := - \displaystyle\sum_{i = 1}^{n}\partial_{x_i}(|\partial_{x_i} u |^{p_{i}(x) - 2} \partial_{x_i} u ).
\end{equation}

Recently, the study of partial differential equations with variable exponents has been motivated by the description of models in electrorheological and thermorheological fluids, image processing, or robotics (see for instance \cite{CLR} for image restoration and the book \cite{DHHR} for further applications and references).

Existence and multiplicity of weak solutions to problems involving the $\pp(x)$-Laplace operator have been considered in several works. In \cite{BPR},  variational methods have been applied  to non-homogeneous problems with a non linearity depending on lower order terms. Also, the critical point theory has been used in \cite{MM}  for more general operators of anisotropic type, and in  \cite{Ji} for eigenvalue problems with both Dirichlet and Neumann   boundary conditions. In these works, the Ambrosetti-Rabinowitz assumption on the right-hand side is the key to stablish  the mountain pass geometry of the  energy functional associated to the differential equation. More recently, anisotropic problems which can be critical in some region of the domain,   have been considered in \cite{OR}, and in \cite{Mh}. In this latter case, the right-hand side has low regularity (it belongs for instance to $L^1$). We also mention that the concentration compactness principle of Lions to treat critical problems has been extended to the anisotropic case in the reference \cite{CCL}. Finally, the relation between different notions of weak solutions and their regularity has been studied recently in \cite{OV} for non-homogeneous equations with the $\pp(x)$-Laplacian operator.

On the other hand, topological degree theory can be one of the most effective tools for demonstrating the existence of solutions in nonlinear equations, even without the Ambrosetti-Rabinowitz condition. As a measure of the algebraic number of solutions of the functional equation $F(x) = h$ for a fixed $h$, the degree has special properties such as existence, normalization, additivity, and homotopy invariance. The most powerful is that the value of the degree does not vary under appropriate perturbations (homotopies), which plays an important role in the study of differential and integral equations.

Brouwer \cite{Br} initiated the construction of a topological degree for continuous mappings on a bounded domain of $\RR^n$. Afterwards, Leray and Schauder \cite{LS} developed the degree theory for compact perturbations of the identity operator  in infinite-dimensional Banach spaces. Brouwer \cite{Bro} introduced a topological degree for nonlinear operators of monotone type in reflexive Banach spaces, where the Galerkin method is used to apply Brouwer's degree. Berkovits \cite{Ber} provided a new construction of the Browder degree, based on the Leray-Schauder degree. From this perspective, a recent  extension of the Leray-Schauder degree for operators of generalized monotone type was studied in \cite{Ber} and for possible unbounded operators in \cite{Kim}. One application of the Degree Theory to double phase problems with variable exponents can be seen in \cite{Ou} for degenerate equations.  We refer the reader to the book \cite{V} for a comprehensive introduction to the matter. 

 In this paper, we investigate the existence of weak solution for a Dirichlet boundary value problem involving the anisotropic $\pp(x)$-Laplacian of the following form 
\begin{equation}\label{main problemss}
\begin{cases}
-\Delta_{\pp(x)}u + |u|^{p_M(x)-2}u = f(x, u, D u)\quad \text{in }\Omega\\u= 0 \quad \text{on }\partial \Omega
\end{cases}
\end{equation}
where 
is anisotropic $\pp(x)$-Laplacian operator defined in \eqref{operator}, $\Omega$ is a smooth bounded domain in $\RR^n$ and $f$ is Carathéodory function. The main challenge when dealing with $-\Delta_{\pp\,(x)}$ is that this operator is highly non-homogeneous due to the influence of each component of the spatial variables and the exponents $p_i(x)$ in each partial derivative. Consequently, these facts will require adapting the usual techniques of topological degree theory to our framework. 

Under suitable nonstandard growth conditions for the function $f$ and using Berkovits' topological degree for a class of $(S_+)$-type operators and the theory of anisotropic Sobolev spaces, we establish the existence of a weak solution for the previous problem. 

We next state the main result of the paper. We will assume that  the function $f: \Omega \times \mathbb{R}\times \mathbb{R}^n \to \mathbb{R}$ satisfies
\begin{itemize}
\item[(f1)] $f$ is a Carath\'eodory function, that is, $x \to f(x, u, \xi)$ is measurable for all $u \in \mathbb{R}$ and all $\xi\in \mathbb{R}^n$, and $(u, \xi)\to f(x, u, \xi)$ is continuous for all $x \in \Omega$.
\item[(f2)] There exist $q \in \mathcal{C} (\overline{\Omega})$, $q(x)>1$, and $k \in L^{q'(x)}(\Omega)$ such that
$$|f(x, u, \xi)|\leq C\left(|k|+|u|^{q(x)-1} + \sum_{i=1}^n |\xi_i|^{r_i(x)}\right),$$where
$$1 < q(x)\leq q^+ < p_m:= \min\left\lbrace p_M^-, p_1^-,..., p_n^- \right\rbrace,$$and $r_i$ are non negative and measurable functions satisfying
$$1 < r_i(x)q'(x)\leq p_i(x)\quad \text{ and } \quad  r_i^+ < \dfrac{p_i^-}{p_m}( p_m-1), \quad \text{for all } i=1, ..., n.$$
\end{itemize}

%\begin{remark}\label{rmk}Observe that assumption $(f2)$ implies that
% $$r_i(x)q'(x)\leq p_i(x), \quad \text{for all }x\in \Omega.$$Indeed,
%\end{remark}

The main result of the paper is the following:
\begin{theorem}\label{existence}Suppose that $p_i \in  \mathcal{C}(\overline{\Omega})$ and $ p^{-}_i> 1$ for all $i=1, ..., n$. Assume $(f1)$ and $(f2)$. Then the Dirichlet problem
\begin{equation}\label{main problems}
\begin{cases}
-\Delta_{\pp(x)}u + |u|^{p_M(x)-2}u = f(x, u, D u)\quad \text{in }\Omega\\u= 0 \quad \text{on }\partial \Omega
\end{cases}
\end{equation}has a weak solution $u\in W_0^{1, \pp(x)}(\Omega)$.
\end{theorem}

Observe that we can set $r_i(x)= q(x)-1$ for all $i$ and $x\in \Omega$, with $q$ satisfying $(f2)$, and all the assumptions for $r_i$ are verified. In particular, if $q$ is constant, we recover the results from \cite{Kim}. 

The paper is organized as follows. In Section \ref{Preliminary section}, we will provide the basic notation, definitions and properties related to the $\pp(x)$-Laplace operator and the functional framework. In Section \ref{degree theory}, we give an introduction to the degree theory of Berkovits and the generalization given in \cite{Kim}. Finally, in Section \ref{Proof of existence}, we stablish some preliminary lemmas and we prove the main result of the paper Theorem \ref{existence}. 

\section{Preliminaries}\label{Preliminary section}

\subsection*{Basic notation}

In what follows, given any real separable reflexive Banach space $X$, $X^{\ast}$ will denote its dual space with dual pairing $\langle \cdot, \cdot \rangle$. Given a nonempty set $\Omega$ of $X$,  $\overline{\Omega}$ and $\partial \Omega$ will denote the closure and the boundary of $\Omega$ in $X$, respectively. Strong (weak) convergence is represented by the symbol $\to$ ($\rightharpoonup$). The space of continuous functions over a closed set $A$ is denoted by $\mathcal{C}(A)$.  Also, we will denote by $C$ universal constants that may vary from line to line  in calculations.

\subsection*{Variable exponent spaces}

In this section we introduce basic definitions and preliminary results concerning  variable exponent spaces.

In what follows, all the scalar variable exponents $p=p(x):\Omega \to \mathbb{R}$ are in the space 
$$L^{\infty}_+(\Omega):=\left\lbrace p :\Omega \to [1, \infty): p \in L^\infty(\Omega)\right\rbrace.$$

%and satisfy the log-H\"{o}lder continuity: there exists $C>0$ such that
% 
%\begin{equation}\label{assumpt H}
%|p(x)-p(y)|\leq C\frac{1}{|\log|x-y||}, \quad \text{for all } \;x, y \in \Omega,\; x\neq y,\,|x-y|\leq 1/2.
%\end{equation}

For $p \in L^{\infty}_+(\Omega)$, we define
$$p^{+}:=\sup_{x \in \Omega}p(x) \quad \text{ and }\quad p^{-}:=\inf_{x \in \Omega}p(x).$$

%Condition \eqref{assumpt H} allows that compactly supported smooth functions are dense in Sobolev spaces of variable exponent (see \cite{Sa}). This in particular permits to define Sobolev spaces with zero boundary values as the closure of $C_0^{\infty}$  under the Sobolev norm (see for instance \cite{H}). 

\subsection{Scalar variable exponent spaces}

We define the Lebesgue variable exponent space as
$$L^{p(x)}(\Omega) := \left\{u: \Omega \to \RR: u \ \mbox{is measurable and} \ \displaystyle\int_{\Omega} |u(x)|^{p(x)} dx < \infty \right\},$$equipped with the Luxemburg norm
$$|u|_{p(x)}:=\inf\bigg\{\lambda>0: \int_\Omega \bigg\vert \dfrac{u(x)}{\lambda}\bigg\vert^{p(x)}\,dx \leq 1\bigg\}.$$
It turns out that  $L^{p(x)}(\Omega)$ is a Banach space, which is also separable. Assuming $p^- > 1$, the dual space of $L^{p(x)}(\Omega)$ is $L^{p'(x)}(\Omega)$, where
$$\displaystyle\frac{1}{p(x)} + \displaystyle\frac{1}{p'(x)} = 1, \quad \text{for all }x\in \Omega,$$and the Lebesgue space $L^{p(x)}(\Omega)$ is reflexive. For these and the next results, we refer the reader to  \cite{DHHR}.

\begin{theorem}[H\"{o}lder's inequality]
Assume that $p, p' \in L^{\infty}_+(\Omega)$. If $u \in L^{p(x)}(\Omega)$ and $v \in
L^{p'(x)}(\Omega)$, then
$$\Big\vert \int_\Omega u v \,dx\Big\vert   \leq \left(\frac{1}{p^{-}}+\frac{1}{(p')^{-}} \right)|u|_{p(x)}|v|_{p'(x)}.$$
\end{theorem}
The next proposition states the relation between norms and modulars. 
\begin{proposition}\label{properties modulo} Let $p\in L^{\infty}_+(\Omega)$ and let
$$\rho_p(u):=\int_\Omega |u|^{p(x)}\,dx, \quad u \in L^{p(x)}(\Omega),$$
be the convex modular associated to $p$. Then the following assertions hold:
\begin{itemize}
\item[(i)] $|u|_{p(x)} <1$ (resp. $=1, >1$) if and only if $\rho_p(u)< 1$ (resp. $=1, >1$);
\item[(ii)] $|u|_{p(x)} >1$ implies $|u|_{p(x)}^{p^{-}} \leq \rho_p(u) \leq |u|_{p(x)}^{p^{+}}$, and $|u|_{p(x)} <1$ implies $|u|_{p(x)}^{p^{+}} \leq \rho_p(u) \leq |u|_{p(x)}^{p^{-}}$;
\item[(iii)] $|u|_{p(x)}  \to 0$ if and only if $\rho_p(u)\to 0$, and $|u|_{p(x)} \to \infty$ if and only if $\rho_p(u)\to \infty$.
\end{itemize}
\end{proposition}
The following result allows us to relate the norms of different Lebesgue variable  exponent spaces (see \cite[Theorem 2.3]{La} for a proof).
\begin{lemma}\label{product}
Let $r$ be a measurable function and $q \in L^{\infty}_+(\Omega)$. Assume that $0 < r^{-}\leq r(x)\leq q(x)$ for all $x\in \Omega$. Then, if $f\in L^{q(\Omega)}$, there holds
$$\min\{|f|_{q(x)}^{r^{+}},|f|_{q(x)}^{r^{-}}\}\leq ||f|^{r(x)}|_{q(x)/r(x)} \leq  \max\{|f|_{q(x)}^{r^{+}},|f|_{q(x)}^{r^{-}}\}.$$
\end{lemma}
Let us denote the distributional gradient of a given function $u$ by $D u$. Then we can define the variable Sobolev space $W^{1, p(x)}(\Omega)$ as
$$W^{1, p(x)}(\Omega):=\left\lbrace u \in L^{p(x)}(\Omega): |D u|\in L^{p(x)}(\Omega)\right\rbrace,$$
equipped with the norm
$$|u|_{1, p(x)}:=|u|_{p(x)}+|D u|_{p(x)},$$
and we denote by $W_0^{1, p(x)}(\Omega)$ the closure of $\mathcal{C}_0^{\infty}(\Omega)$ in $W^{1,
p(x)}(\Omega)$.  If $p^{-}>1$, then $W^{1, p(x)}(\Omega)$ and $W_0^{1, p(x)}(\Omega)$ are reflexive and separable Banach spaces. 

%Notice that, due to the log-H\"{o}lder condition \eqref{assumpt H}, $\mathcal{C}_0^{\infty}(\Omega)$ is dense in $W^{1, p(x)}(\Omega)$. 
%
%The following 
% Embedding Theorem holds (see for instance \cite{FZ4} and \cite{FSZ}):
%
%\begin{theorem}
%Let $\Omega\subset \mathbb{R}^n$ be bounded and open. Assume that  $p^{+}< n$, $p, q \in C(\overline{\Omega})$, with
%$$1 < p^{-}\leq q(x)< \dfrac{np(x)}{n-p(x)}, \quad x \in \overline{\Omega},$$then the embedding $W^{1, p(x)}(\Omega)\hookrightarrow L^{q(x)}(\Omega)$ is continuous and compact.
%\end{theorem} 

\subsection{Anisotropic variable exponent spaces} Next, we turn  our attention to anisotropic Sobolev spaces. Let $\pp:\Omega\to \mathbb{R}^{n}$ denote the vector field
$$\pp(x)=(p_1(x),..., p_n(x)), \quad x \in \Omega,$$where
$$p_i\in L^{\infty}_+(\Omega), \quad \text{for all }i.$$
We define
$$p_M(x):=\sup\left\lbrace p_1(x), ..., p_n(x)\right\rbrace.$$Then, $p_M\in L^{\infty}_+(\Omega).$ The anisotropic variable exponent Sobolev space $W^{1, \pp(x)}(\Omega)$ is defined as
$$W^{1, \pp(x)}(\Omega):=\left\lbrace u \in L^{p_M(x)}(\Omega): \partial_{x_i}u \in L^{p_i(x)}(\Omega),\,i=1,..., n \right\rbrace,$$endowed with the norm
$$|u|_{\pp(x)}:=|u|_{p_M(x)}+\sum_{i=1}^{n}|\partial_{x_i}u|_{p_i(x)}.$$
Observe that the space $W^{1, \pp(x)}(\Omega)$ may be defined equivalently by
$$W^{1, \pp(x)}(\Omega):=\left\lbrace u \in L^{1}_{loc}(\Omega): u\in L^{p_i(x)}(\Omega),\, \partial_{x_i}u \in L^{p_i(x)}(\Omega),\,i=1,..., n \right\rbrace.$$Without any further assumption on $\pp$, $W^{1, \pp(x)}(\Omega)$ is a Banach space (\cite[Theorem 2.1]{Fan2011}). Moreover, if $p_i^{-}>1$, then by \cite[Theorem 2.3]{Fan2011}, the space $W^{1, \pp(x)}(\Omega)$ is  reflexive.  Also, to consider zero boundary condition, define $W_0^{1, \pp(x)}(\Omega)$ as the closure of $C_0^{\infty}(\Omega)$ with respect to the norm in $W^{1, \pp(x)}(\Omega)$.

\begin{remark}Zero boundary conditions may be considered in the trace sense. Indeed, let $\Omega$ be a bounded domain with Lipschitz boundary and for $u \in W^{1, \pp(x)}(\Omega) \subset W^{1, 1}(\Omega)$, we let $u |_{\partial \Omega}$ the boundary trace of $u$ in $W^{1, 1}(\Omega)$. Then, we can define
$$W=\left\lbrace u \in W^{1, \pp(x)}(\Omega): u |_{\partial \Omega}=0 \right\rbrace.$$In the variable exponent case, we can have $W \neq W_0^{1, \pp(x)}(\Omega)$. However, when for all $i$, $p_i$ satisfies the log-H\"{o}lder continuity, that is, there exists $C>0$ such that 
\begin{equation*}
|p_i(x)-p_i(y)|\leq C\frac{1}{|\log|x-y||}, \quad \text{for all } \;x, y \in \Omega,\; x\neq y,\,|x-y|\leq 1/2,
\end{equation*}then by \cite[Theorem 2.4]{Fan2011}, the set $\mathcal{C}_0^{\infty}(\Omega)$ is dense in $W^{1, \pp(x)}(\Omega)$ and so  $W = W_0^{1, \pp(x)}(\Omega)$.
\end{remark}

We end the section by quoting the following embedding result from \cite[Theorem 2.5]{Fan2011}:

\begin{theorem}
Let $\Omega\subset \mathbb{R}^n$ be any bounded domain,  $p_i, q \in L^{\infty}_{+}(\Omega)\cap \mathcal{C}(\overline{\Omega})$ for all $i=1, ..., n$, such that 
$$1 < q(x)< \max\left\lbrace p_M(x), \overline{p}^*(x) \right\rbrace,\quad for\,all\,x\in \overline{\Omega},$$where
\begin{equation}
\overline{p}^*(x):=\begin{cases} \dfrac{n \overline{p}(x) }{n-\overline{p}(x)}, \,if\,\,\overline{p}(x) < n,\\
\infty,\,if\,\,\overline{p}(x) \geq n
\end{cases}
\end{equation}
and
\begin{equation}\label{expo}
\overline{p}(x):=\dfrac{n}{\displaystyle\sum_{i=1}^n\dfrac{1}{p_i(x)}},
\end{equation}
then we have that the embedding 
$$W_0^{1, \pp(x)}(\Omega)\hookrightarrow L^{q(x)}(\Omega)$$is compact.
\end{theorem}

  Therefore, Problem \eqref{main problems} could be critical in the region where 
$$\overline{p}^*(x) \leq p_M(x).$$Observe that this is never the case when $p$ is a scalar function.

Some further anisotropic Sobolev embedding results in the case $\overline{p}^*(x) > p_M(x)$ for all $x\in \overline{\Omega}$ may be found in \cite{La} and the reference therein.

%The space  $W^{1, \pp(\cdot)}(\Omega)$ is a reflexive Banach space since the operator $T:  W_0^{1, \pp(\cdot)}(\Omega) \to L^{p_1(\cdot)}\times \cdots \times  L^{p_N(\cdot)}$ given by
%$$T(u)=D u$$is an isometry and so $T( W_0^{1, \pp(\cdot)}(\Omega))$ is a closed subspace of $ L^{p_1(\cdot)}\times \cdots \times  L^{p_N(\cdot)}$. 

%We introduce next some useful constant exponents. Let
%\begin{equation*}
%\textbf{P}_{+}=((p_1)_{+}, ..., (p_N)_{+}), \quad \textbf{P}_{-}=((p_1)_{-}, ..., (p_N)_{-}).
%\end{equation*}Also, define
%\begin{equation}\label{exponent constant}
%\textbf{P}_{+}^{+}=\max\{(p_1)_{+}, ..., (p_N)_{+}\}, \,\, \textbf{P}_{+}^{-}=\max\{(p_1)_{-}, ..., (p_N)_{-}\},\,\, \textbf{P}_{-}^{-}=\min\{(p_1)_{-}, ..., (p_N)_{-}\}.
%\end{equation}Assume
%\begin{equation}\label{Hipotesis p}
%\sum_{i=1}^{N}\dfrac{1}{p_i{-}} > 1,
%\end{equation}and  define the scalar exponent
%$$P_{-}^{*}:=\dfrac{N}{\sum_{i=1}^{N}1/p_{i}^{-}-1}, \quad P_{-, \infty}:=\max\left\lbrace P_{-}^+, P_{-}^*\right\rbrace.$$
%The following compactness results is taken from \cite[Theorem 1]{DHHR}.
%\begin{theorem}
%Assume that \eqref{Hipotesis p} holds and that $q \in \mathcal{C}_{+}(\overline{\Omega})$ satisfies $q(x)< P_{-, \infty}$ for all $x \in \overline{\Omega}.$ Then the embedding
%$$W_0^{1,\pp(\cdot)}(\Omega) \hookrightarrow L^{q(\cdot)}(\Omega)$$is continuous and compact.
%\end{theorem}
\section{Preliminaries in Degree Theory}\label{degree theory}

In this section, we will provide a short review of the Degree Theory stated by Berkovits in \cite{Ber} and its generalization \cite{Kim}. We start by providing some definitions.

\begin{definition}
\, Let $Y$ be a real Banach space. A operator $F: \Omega \subset X \to Y$ is said to be
\begin{enumerate}
    \item bounded, if it takes any bounded set into a bounded set.
    \item demicontinuous, if for any sequence $(u_n)_{n\in \mathbb{N}} \subset \Omega$, $u_n \to u$ implies $F(u_n) \rightharpoonup F(u)$.
    \item compact, if it is continuous and the image of any bounded set is relatively compact.
\end{enumerate}
\end{definition}

\begin{definition}
\, A mapping $F: \Omega \subset X \to X^{\ast}$ is said to be
\begin{enumerate}
    \item of class $(S_+)$, if for any sequence $(u_n)_{n\in \mathbb{N}} \subset \Omega$ with $u_n \rightharpoonup u$ and $\limsup_{n\to \infty} \langle Fu_n, u_n - u \rangle \leqslant 0 $ we have $u_n \to u$.
    \item quasimonotone, if for any sequence $(u_n)_{n\in \mathbb{N}} \subset \Omega$ with $u_n \rightharpoonup u$ we have $\limsup_{n\to \infty} \langle Fu_n, u_n - u \rangle \geqslant 0 $.
\end{enumerate}
\end{definition}

Observe that if $F$ is compact, then it is quasimonotone.

\begin{definition}
\, Let $T: \Omega_1 \subset X \to X^{\ast}$ be a bounded operator such that $\Omega \subset \Omega_1$. For any operator $F: \Omega \subset X \to X$, we say that 
\begin{enumerate}
    \item $F$ is class $(S_+)_T$, if for any sequence $(u_n)_{n\in \mathbb{N}} \subset \Omega$ with $u_n \rightharpoonup u$, $y_n:= Tu_n \rightharpoonup y$ and $\limsup_{n\to \infty} \langle Fu_n, y_n - y \rangle \leqslant 0 $, we have $u_n \to u$.
    \item $F$ has the property $(QM)_{T}$, if for any sequence $(u_n)_{n\in \mathbb{N}} \subset \Omega$ with $u_n \rightharpoonup u$, $y_n:= Tu_n \rightharpoonup y$, we have $\limsup_{n\to \infty} \langle Fu_n, y_n - y \rangle \geqslant 0. $
\end{enumerate}
\end{definition}

In the sequel, we consider the following classes of operators. We let $\mathcal{O}$ the collection of all bounded open sets in $X$:
\begin{equation*}
    \begin{split}
        \mathcal{F}_1(\Omega) &:= \bigg\{F: \Omega \to X^{\ast}: \, F \, \mbox{is bounded, demicontinuous and of class} \, (S_+) \bigg\}, \\
        \mathcal{F}_{T,B}(\Omega) &:= \bigg\{F: \Omega \to X: \, F \, \mbox{is bounded, demicontinuous and of class} \, (S_+)_{T} \bigg\}, \\
        \mathcal{F}_{B}(X) &:= \bigg\{F \in \mathcal{F}_{T,B}(\overline{E}): \, E \in \mathcal{O},  T \in \mathcal{F}_1(\overline{E}) \bigg\}, \\
        \mathcal{F}_{T}(\Omega) &:= \bigg\{F: \Omega \to X: \, F \, \mbox{is demicontinuous and of class} \, (S_{+})_T \bigg\}, \\
    \end{split}
\end{equation*}
for any $\Omega \subset \mbox{D}(F)$, where $\mbox{D}(F)$ denotes the domain of $F$, and any $T \in \mathcal{F}_1(\Omega)$. Now, we define 
$$\mathcal{F}(X):= \bigg \{F \in \mathcal{F}_{T}(\overline{E}): E \in \mathcal{O}, T \in \mathcal{F}_1(\overline{E}) \bigg\},$$
when $f \in $ $T \in \mathcal{F}_1(\overline{E})$ to some $T$, we call $T$ an essential inner map to $F$.

\begin{lemma}\label{lemma KH}
\, [\cite{Kim}, Lemma 2.3] Let $T \in \mathcal{F}_1 (\overline{E})$ be continuous and $S : \mbox{D}(S) \subset X^{\ast} \to X$ be demicontinuous,
such that $T(\overline{E}) \subset \mbox{D}(S)$, where $E$ is a bounded open set in a real reflexive Banach space $X$. Then, the
following statements are true:
\begin{enumerate}
    \item If $S$ is quasimonotone, then $I + S \circ T \in \mathcal{F}_{T}(\overline{E})$, where $I$ denotes the identity operator.
    \item If $S$ of class $(S_+)$, then $S \circ T \in \mathcal{F}_{T}(\overline{E})$.
\end{enumerate}
\end{lemma}

\begin{definition}
     Suppose that $E$ is a bounded open subset of a real reflexive Banach space $X$, $T \in \mathcal{F}_1(\overline{E})$ is
continuous and let $F,S \in F_T (\overline{E})$.
The \textit{affine homotopy} $H : [0, 1] \times E \to X$ defined by $H(t,u):= (1-t) Fu + t Su,$  for all $(t,u) \in [0,1] \times E$ is called an admissible affine homotopy with the common continuous essential inner map $T$.
\end{definition}
In \cite{Kim} Lemma 2.5, it is proved that the above affine homotopy is of class $(S_+)_T$. Next, following \cite{Kim}, we give the topological degree for the class $\mathcal{F}(X)$.

\begin{theorem}\label{main theorem degree}
\, Let
$M = 
\{(F, E, h) : E \in \mathcal{O}, T \in \mathcal{F}_1(E), F \in \mathcal{F}_{T}(E), h \notin  F(\partial E)\}$.
Then, there exists a unique degree function $d : M \to \ZZ$ that satisfies the following properties:
\begin{enumerate}
    \item (Normalization).  $d(I,E,h) = 1$, for all $h \in E$.
    \item (Additivity).  Let $F \in \mathcal{F}_{T}(E)$. If  $E_1$ and $E_2$ are two disjoint open subsets of $E$, such that $h \notin F(\overline{E} \setminus (E1 \cup E2))$. Then we have
$$d(F, E, h) = d(F, E1, h) + d(F, E2, h).$$
    \item (Homotopy invariance).  If $H : [0, 1] \times \overline{E} \to X$ is a bounded admissible affine homotopy with a common
continuous essential inner map and $h: [0, 1] \to X$ is a continuous path in $X$, such that $h(t) \notin H(t, \partial E)$, for all $t \in [0, 1]$, then $d(H(t, \cdot), E, h(t))$ is constant, for all $t \in [0, 1]$.
    \item (Existence). \, If $d(F, E, h)\neq 0$, then the equation $Fu = h$ has a solution in $E$.
\end{enumerate}
\end{theorem}

\begin{definition}
    \, The above degree $d$ is defined as follows:
    \begin{equation*}
        d(F,E,h):= d_{B}(F|_{\overline{E_0}}, E_0,h),
    \end{equation*}
    where $d_B$ its the Berkovits degree \cite{Ber} and $E_0$ is any open subset of $E$ with $F^{-1}(h) \subset E_0$ and $F$ is bounded on $\overline{E}_0$ (see \cite[Corollary 2.8]{Kim}).
\end{definition}

\section{Proof of Theorem \ref{existence}}\label{Proof of existence}

We start with some preliminary lemmas.

\begin{lemma}\label{lemma S}
Assume that $(f1)$ and $(f2)$ hold. Then, the operator $S: W_0^{1, \pp(x)}(\Omega)\to W^{-1, \pp'(x)}(\Omega)$ defined by
$$\left\langle Su, v\right\rangle = -\int_\Omega f(x, u, D u)v\,dx$$is compact.   
\end{lemma}

\begin{proof}
We will divide the proof in four stages.

\textbf{Step 1:} the operator $S$ is well-defined. Indeed, let $u, v \in W_0^{1, \pp(x)}(\Omega)$. First, observe that under assumption $(f1)$, we have
$$k+|u|^{q(x)-1} + \sum_{i=1}^n |\partial_{x_i}u|^{r_i(x)}\in L^{q'(x)}(\Omega).$$In fact, 
\begin{itemize}
\item $k \in L^{q'(x)}(\Omega)$ by assumption.
\item Recalling that $q(x)\leq p_M(x)$ and Proposition \ref{properties modulo}, we get
\begin{equation}
\begin{split}
\int_\Omega |u|^{(q(x)-1)q'(x)}\,dx = \int_\Omega |u|^{q(x)}\,dx \leq |u|_{q(x)}^{q^{-}}+ |u|_{q(x)}^{q^{+}} \leq C\left( |u|_{p_M(x)}^{q^{-}}+ |u|_{p_M(x)}^{q^{+}}\right)<\infty.
\end{split}
\end{equation}
\item Similarly, observing that $1\leq r_i(x)q'(x)\leq p_i(x)$ for $i=1, ..., n$, we derive
$$\int_\Omega |\partial_{x_i}u|^{r_i(x)q'(x)}\,dx \leq C\left( |\partial_{x_i}u|_{p_i(x)}^{\alpha_i^{-}}+ |\partial_{x_i}u|_{p_i(x)}^{\alpha_i^{+}}\right)<\infty,$$where $\alpha_i(x)= r_i(x)q'(x)$.
\end{itemize}Thus, by $(f2)$, we obtain $f(x, u, D u)\in L^{q'(x)}(\Omega)$. Therefore, 
\begin{equation}
\begin{split}
|\left\langle Su, v \right\rangle| \leq \int_\Omega |f(x, u, D u)||v|\,dx \leq C|f(x, u, D u)|_{q'(x)}|v|_{q(x)} \leq  C|f(x, u, D u)|_{q'(x)}|v|_{1, \pp(x)},
\end{split}
\end{equation}where we have used that $q \leq p_M$. Thus, $S$ is well-defined.

\textbf{Step 2:} let $\phi: W_0^{1, \pp(x)}(\Omega)\to L^{q'(x)}(\Omega)$ be the operator defined by
$$\phi u(x):=-f(x, u(x), D u(x)), \quad u \in  W_0^{1, \pp(x)}(\Omega).$$We next show that $\phi$ is well-defined, sends bounded set into bounded sets (that is, $\phi$ is a bounded operator), and it is continuous.

Indeed, by the previous step we know that $f(x, u(x), D u(x)) \in L^{q'(x)}(\Omega)$ so $\phi$ is well-defined. Moreover, the previous step also implies that
$$|\phi u|_{q'(x)} \leq C\left(1+ |u|_{1, \pp(x)}^{q^{+}}+|u|_{1, \pp(x)}^{q^{-}}+\sum_{i=1}^n |u|_{1, \pp(x)}^{\alpha_i^{+}}+\sum_{i=1}^n |u|_{1, \pp(x)}^{\alpha_i^{-}}\right),$$so in particular, $\phi$ is a bounded operator. 

To show that $\phi$ is continuous, let $u_k\to u$ in $W_0^{1, \pp(x)}(\Omega)$. Then, $u_k\to u$ in $L^{p_M(x)}(\Omega)$ and $\partial_{x_i}u_k\to \partial_{x_i}u$ in $L^{p_i(x)}(\Omega)$ for all $i$. Take any subsequence $u_{k_l}$ of $u_k$, that we will still denote by $u_k$.  Hence, there exist functions $v \in L^{p_M(x)}(\Omega)$ and $w_i \in L^{p_i(x)}(\Omega)$, $i=1, ..., n$, such that for a further subsequence, still denoted by $u_k$, there holds
$$|u_k|\leq v(x) \quad \text{ and }\quad |\partial_{x_i}u_k |\leq w_i(x), \quad i=1,..., n,$$for a. e. $x\in \Omega$. Thus,  by assumption $(f2)$,
$$f(x, u_k, D u_k)\leq C\left(|k| + v(x)^{q(x)-1}+\sum_{i=1}^nw^{r_i(x)}_i(x) \right)\in L^{q'(x)}(\Omega)$$and
$$|f(x, u_k, D u_k)-f(x, u, D u)|^{q'(x)}\to 0\quad a.e. \text{ in }\Omega.$$Thus, by Lebesgue Theorem,
$$\int_\Omega|f(x, u_k, D u_k)-f(x, u, D u)|^{q'(x)}\,dx\to 0, \quad \text{as }k\to \infty.$$Since this holds for any subsequence of $u_k$, we prove that $\phi$ is continuous. 

\textbf{Step 3:} $S$ is compact. Observe that $S= I^{*}\circ \phi$, where $I^{*}$ is the adjoint operator of the inclusion $I: W^{1, \pp(x)}(\Omega)\to L^{q(x)}(\Omega)$ (which is compact). Hence, we conclude that $S$ is compact. 
\end{proof}

\begin{lemma}\label{lemma f}
The mapping $F: W_0^{1, \pp(x)}(\Omega)\to W_0^{-1, \pp'(x)}(\Omega)$ defined by
$$\left\langle Fu, v\right\rangle :=\sum_{i=1}^n\int_\Omega |\partial_{x_i}u|^{p_i(x)-2}\partial_{x_i}u\,\partial_{x_i}v\,dx + \int_\Omega|u|^{p_M(x)-2}uv\,dx$$for all $u, v \in  W_0^{1, \pp(x)}(\Omega)$, is  bounded, coercive, continuous, strictly monotone and an $(S_+)$-operator. 
\end{lemma}

\begin{proof}
We start proving that $F$ is bounded. Let $u, v \in  W_0^{1, \pp(x)}(\Omega)$. Then,
\begin{equation}
\begin{split}
|\left\langle Fu, v \right\rangle| &\leq \sum_{i=1}^n \int_\Omega |\partial_{x_i}u|^{p_i(x)-1}|\partial_{x_i}v|\,dx +\int_{\Omega}|u|^{p_M(x)-1}|v|\, dx \\&\leq  C\left(\sum_{i=1}^n ||\partial_{x_i}u|^{p_i(x)-1}|_{p_i'(x)}|\partial_{x_i}v|_{p_i(x)} + ||u|^{p_M(x)-1}|_{p_M'(x)}|v|_{p_M(x)} \right).
\end{split}
\end{equation}Applying Lemma  \ref{product}, we conclude that
\begin{equation*}
\begin{split}
|\left\langle Fu, v \right\rangle| &\leq C|v|_{1, \pp(x)}\left( \sum_{i=1}^n \max\left\lbrace |\partial_{x_i}u|_{p_i(x)}^{p^+_1-1},|\partial_{x_i}u|_{p_i(x)}^{p^-_1-1} \right\rbrace + \max\left\lbrace |u|_{p_M(x)}^{p_M^+ -1}, |u|_{p_M(x)}^{p_M^- -1} \right\rbrace\right)
\end{split}
\end{equation*}This shows that $F$ sends bounded sets into bounded sets. 

By a  similar reasoning as in the proof of Step 2 in the previous lemma, we conclude that $F$ is continuous. 

% before shows that if $u_n \to u$ in $ W_0^{-1, \pp'(x)}(\Omega)$, then 
%\begin{equation*}
%\begin{split}
%&|\left\langle Fu_n-Fu, v \right\rangle|\\ &\quad \leq C|v|_{1, \pp(x)}\left( \sum_{i=1}^n \max\left\lbrace |\partial_{x_i}(u_n-u)|_{p_i(x)}^{p^+_1-1},|\partial_{x_i}(u_n-u)|_{p_i(x)}^{p^-_1-1} \right\rbrace + \max\left\lbrace |u|_{p_M(x)}^{p_M^+ -1}, |u_n-u|_{p_M(x)}^{p_M^- -1} \right\rbrace\right)
%\end{split}
%\end{equation*}Hence,

We next show that $F$ is an $(S_+)$-operator. Let $u_n \rightharpoonup u$ in $ W_0^{1, \pp(x)}(\Omega)$ and assume that
\begin{equation}\label{asump S}
\limsup_{n\to \infty}\left\langle Fu_n, u_n-u\right\rangle\leq 0.
\end{equation}First, the weak convergence implies that \eqref{asump S} may be written as
\begin{equation}\label{asump SS}
\left\langle Fu_n, u_n-u\right\rangle = \left\langle Fu_n-Fu, u_n-u\right\rangle + \left\langle Fu, u_n-u\right\rangle = \left\langle Fu_n-Fu, u_n-u\right\rangle + o_n(1).
\end{equation}

%\color{red}
We will appeal to the following well-known inequalities:
\begin{equation}\label{ineq degenerate}
(|\xi|^{p-2}\xi - |\eta|^{p-2}\eta)(\xi-\eta)\geq C|\xi-\eta|^{p}, \quad p \geq 2,
\end{equation}and
\begin{equation}\label{ineq singular}
(|\xi|^{p-2}\xi - |\eta|^{p-2}\eta)(\xi-\eta)\geq (p-1)\dfrac{|\xi -\eta|^2}{(|\xi|+|\eta|)^{2-p}}, \quad 1< p < 2.
\end{equation}
Let also consider the following sets
$$\Omega_{1, i}:=\left\lbrace x\in \Omega: p_i(x)\geq 2\right\rbrace$$
and
$$\Omega_{2, i}:=\left\lbrace x\in \Omega: 1 < p_i(x)< 2\right\rbrace,$$for $i=1, ..., n$, and
$$\Omega_{1, M}:=\left\lbrace x\in \Omega: p_M(x)\geq 2\right\rbrace,$$
$$\Omega_{2, M}:=\left\lbrace x\in \Omega: 1 < p_M(x)< 2\right\rbrace.$$
Then, in the next calculations we will use the inequality \eqref{ineq degenerate} over the sets $\Omega_{1, (\cdot)}$, and inequality \eqref{ineq singular} in $\Omega_{2, (\cdot)}$. We start with the degenerate case.  Let $i$ such that $\Omega_{1, i}\neq \emptyset$. Then by \eqref{ineq degenerate}, we obtain
\begin{equation}\label{ineq Ff4}
\begin{split}
& \int_{\Omega_{1, i}} \left(|\partial_{x_i}u_n|^{p_i(x)-2}\partial_{x_i}u_n-|\partial_{x_i}u|^{p_i(x)-2}\partial_{x_i}u\right)\left(\partial_{x_i}u_n-\partial_{x_i}u\right)\,dx  \geq C\int_{\Omega_{1, i}} |\partial_{x_i}(u_n-u)|^{p_i(x)}\,dx.
\end{split}
\end{equation}If $\Omega_{1, M}\neq \emptyset$, then by \eqref{ineq degenerate}
\begin{equation}\label{ineq Ff5}
\int_{\Omega_{1, M}}\left(|u_n|^{p_M(x)-2}u_n-|u|^{p_M(x)-2}u\right)(u_n-u)\,dx \geq C\int_{\Omega_{1, i}} |u_n-u|^{p_M(x)}\,dx.
\end{equation}Next, we consider singular cases. Suppose that $\Omega_{2, i}\neq \emptyset$ for some $i$ fixed.  Then by H\"{o}lder inequality
\begin{equation}\label{ineqq Ff1}
\begin{split}
& \int_{\Omega_{2, i}}|\partial_{x_i}u_n-\partial_{x_i}u|^{p_i(x)}\,dx\\ & \quad  = \int_{\Omega_{2, i}}\dfrac{|\partial_{x_i}u_n-\partial_{x_i}u|^{p_i(x)}}{(|\partial_{x_i}u_n|+|\partial_{x_i}u|)^{(2-p_i(x))p_i(x)/2}}.\left(|\partial_{x_i}u_n|+|\partial_{x_i}u|\right)^{(2-p_i(x))p_i(x)/2}\,dx\\ & \quad\leq C\bigg| \dfrac{|\partial_{x_i}u_n-\partial_{x_i}u|^{p_i(x)}}{(|\partial_{x_i}u_n|+|\partial_{x_i}u|)^{(2-p_i(x))p_i(x)/2}}\bigg|_{2/p_i(x), \Omega_{2, i}}.\bigg| \left(|\partial_{x_i}u_n|+|\partial_{x_i}u|\right)^{(2-p_i(x))p_i(x)/2}\bigg|_{2/(2-p_i(x)), \Omega_{2, i}}.
\end{split}
\end{equation}
Observe that by Lemma 2.3 and letting $\alpha_i(x)= (2-p_i(x))p_i(x)/2$,  we obtain
\begin{equation}
\begin{split}
&\bigg| \left(|\partial_{x_i}u_n|+|\partial_{x_i}u|\right)^{(2-p_i(x))p_i(x)/2}\bigg|_{2/(2-p_i(x)), \Omega_{2, i}} \\ & \qquad\leq  \max\left\lbrace  (|\partial_{x_i}u_n|+|\partial_{x_i}u|_{p_i(x), \Omega_{2, i}})^{\alpha_i^+},(|\partial_{x_i}u_n|+|\partial_{x_i}u|_{p_i(x), \Omega_{2, i}})^{\alpha_i^-} \right\rbrace \\ & \qquad \leq \max\left\lbrace  (|u_n|_{1, \pp(x), \Omega_{2, i}}+|u|_{1, \pp(x), \Omega_{2, i}})^{\alpha_i^+},(|u_n|_{1, \pp(x), \Omega_{2, i}}+|u|_{1, \pp(x), \Omega_{2, i}})^{\alpha_i^-} \right\rbrace.
\end{split}
\end{equation}For further reference, we set
\begin{equation}\label{new constant}
D_i(u_n, u):=\max\left\lbrace  (|u_n|_{1, \pp(x)}+|u|_{1, \pp(x)})^{\alpha_i^+},(|u_n|_{1, \pp(x)}+|u|_{1, \pp(x)})^{\alpha_i^-} \right\rbrace.
\end{equation}Since $u_n \rightharpoonup u$, we get that $D_i(u_n, u)\leq C$ for all $n$ and some  constant $C> 0$.

Now, by Proposition \ref{properties modulo} and inequality \eqref{ineq singular}, we get
\begin{equation}\label{ineqq Ff}
\begin{split}
&\bigg\| \dfrac{|\partial_{x_i}u_n-\partial_{x_i}u|^{p_i(x)}}{(|\partial_{x_i}u_n|+|\partial_{x_i}u|)^{(2-p_i(x))p_i(x)/2}}\bigg\|_{2/p_i(x), \Omega_{2, i}} \\ & \quad\leq \max\left\lbrace  \left( \int_{\Omega_{2, i}}\dfrac{|\partial_{x_i}u_n-\partial_{x_i}u|^{2}}{(|\partial_{x_i}u_n|+|\partial_{x_i}u|)^{2-p_i(x)}}\,dx\right)^{p_i^+/2},\left( \int_{\Omega_{2, i}}\dfrac{|\partial_{x_i}u_n-\partial_{x_i}u|^{2}}{(|\partial_{x_i}u_n|+|\partial_{x_i}u|)^{2-p_i(x)}}\,dx\right)^{p_i^-/2}  \right\rbrace\\ & \quad\leq C \max\bigg\{  \left( \sum_{k=1}^n\int_{\Omega_{2, k}}\dfrac{|\partial_{x_k}u_n-\partial_{x_k}u|^{2}}{(|\partial_{x_k}u_n|+|\partial_{x_k}u|)^{2-p_k(x)}}\,dx + \int_{\Omega_{2, M}}\dfrac{|u_n-u|^{2}}{(|u_n|+|u|)^{2-p_M(x)}}\,dx \right)^{p_i^+/2},\\& \quad \left( \sum_{k=1}^n\int_{\Omega_{2, k}}\dfrac{|\partial_{x_k}u_n-\partial_{x_k}u|^{2}}{(|\partial_{x_k}u_n|+|\partial_{x_k}u|)^{2-p_k(x)}}\,dx + \int_{\Omega_{2, M}}\dfrac{|u_n-u|^{2}}{(|u_n|+|u|)^{2-p_M(x)}}\,dx \right)^{p_i^-/2}\bigg\}\\& \quad \leq C  \max\bigg\{\bigg(\sum_{k=1}^n\int_{\Omega_{2, k}} \left(|\partial_{x_k}u_n|^{p_k(x)-2}\partial_{x_k}u_n-|\partial_{x_k}u|^{p_k(x)-2}\partial_{x_k}u\right)\left(\partial_{x_k}u_n-\partial_{x_k}u\right)\,dx \\ & \quad + \int_{\Omega_{2, M}}\left(|u_n|^{p_M(x)-2}u_n-|u|^{p_M(x)-2}u\right)(u_n-u)\,dx\bigg)^{p_i^+/2},\\& \quad \bigg(\sum_{k=1}^n\int_{\Omega_{2, k}} \left(|\partial_{x_k}u_n|^{p_k(x)-2}\partial_{x_k}u_n-|\partial_{x_k}u|^{p_k(x)-2}\partial_{x_k}u\right)\left(\partial_{x_k}u_n-\partial_{x_k}u\right)\,dx \\ & \quad + \int_{\Omega_{2, M}}\left(|u_n|^{p_M(x)-2}u_n-|u|^{p_M(x)-2}u\right)(u_n-u)\,dx\bigg)^{p_i^-/2}\bigg\} \\ & \quad \leq C \max\bigg\{\left\langle Fu_n-Fu, u_n-u\right\rangle^{p_i^+/2},  \left\langle Fu_n-Fu, u_n-u\right\rangle^{p_i^-/2}\bigg\} ,
\end{split}
\end{equation}where we have also used \eqref{ineq Ff4} and \eqref{ineq Ff5} to recover the full expression of $\left\langle Fu_n-Fu, u_n-u\right\rangle$. Observe that the above estimates show that
$$ \left\langle Fu_n-Fu, u_n-u\right\rangle \geq 0, \text{ for all }n,$$which together with \eqref{asump SS} imply that
$$\left\langle Fu_n-Fu, u_n-u\right\rangle \to 0$$as $n\to \infty$. A similar reasoning is applied if $\Omega_{2, M}\neq \emptyset$. Observe that this yields that all the integrals:
$$\int_{\Omega_{1, i}}|\partial_{x_i}u_n-\partial_{x_i}u|^{p_i(x)}\,dx, \int_{\Omega_{1, M}}|u_n-u|^{p_M(x)}\,dx, \int_{\Omega_{2, i}}|\partial_{x_i}u_n-\partial_{x_i}u|^{p_i(x)}\,dx, \int_{\Omega_{2, M}}|u_n-u|^{p_M(x)}\,dx$$converges to $0$ as $n\to \infty$. Hence, $u_n\to u$ in $W_0^{1, \pp(x)}(\Omega)$. This proves that $F$ is an $(S_+)$-operator. 
We proceed further to show that $F$ is strictly monotone.
Let
$$I_d= \left\lbrace i: \Omega_{1, i}\neq \emptyset \right\rbrace$$
and
$$I_s= \left\lbrace i: \Omega_{2, i}\neq \emptyset \right\rbrace.$$We assume for simplicity that $\Omega_{k, M}\neq \emptyset$ for $k=1, 2.$
 Combining  \eqref{ineq Ff4}, \eqref{ineq Ff5} and \eqref{ineqq Ff}, we derive for $u, v \in W_0^{1, \pp(x)}(\Omega)$  with $u \neq v$ in $\Omega$ that 
\begin{equation}
\begin{split}
&\left\langle Fv-Fu, v-u\right\rangle  \geq  C\bigg(\sum_{i\in I_d}^n\int_{\Omega_{1, i}} |\partial_{x_i}(v-u)|^{p_i(x)}\,dx +\int_{\Omega_{1, M}} |v-u|^{p_M(x)}\,dx \\ & \quad+\sum_{i\in I_s} \min\bigg\{\left(\dfrac{1}{D_i(v, u)} \int_{\Omega_{2, i}}|\partial_{x_i}v-\partial_{x_i}u|^{p_i(x)}\,dx\right)^{2/p_i^-}, \left(\dfrac{1}{D_i(v, u)}\int_{\Omega_{2, i}}|\partial_{x_i}v-\partial_{x_i}u|^{p_i(x)}\,dx\right)^{2/p_i^+}\bigg\}\\& \quad + \min\bigg\{\dfrac{1}{D_M(v, u)}\left(\int_{\Omega_{2, M}} |v-u|^{p_M(x)}\,dx \right)^{2/p_M^-}, \left(\dfrac{1}{D_M(v, u)}\int_{\Omega_{2, M}} |v-u|^{p_M(x)}\,dx \right)^{2/p_M^+}\bigg\}\bigg)>0,
\end{split}
\end{equation}where the constants $D_i(\cdot, \cdot)$ is defined in \eqref{new constant} and where $D_M(\cdot, \cdot)$ is the analogous constant for the exponent $p_M$. Observe that none of these constant is zero. Hence, $F$ is strictly monotone. 
%\normalcolor

%By appealing to the inequality
%$$(|\xi|^{p-2}\xi - |\eta|^{p-2}\eta)(\xi-\eta)\geq c|\xi-\eta|^{p}, \quad p \geq 2,$$we get
%\begin{equation}\label{long S}
%\begin{split}
%\left\langle Fu_n-Fu, u_n-u\right\rangle & =\sum_{i=1}^n\int_\Omega \left(|\partial_{x_i}u_n|^{p_i(x)-2}\partial_{x_i}u_n-|\partial_{x_i}u|^{p_i(x)-2}\partial_{x_i}u\right)\left(\partial_{x_i}u_n-\partial_{x_i}u\right)\,dx \\ & + \int_\Omega\left(|u_n|^{p_M(x)-2}u_n-|u|^{p_M(x)-2}u\right)(u_n-u)\,dx\\ & \geq c\left(\sum_{i=1}^n\int_\Omega |\partial_{x_i}(u_n-u)|^{p_i(x)}\,dx + \int_\Omega|u_n-u|^{p_M(x)}\,dx \right)\\ & \geq c\bigg(\sum_{i=1}^n\min\left\lbrace |\partial_{x_i}(u_n-u)|^{p_i^+}_{p_i(x)},|\partial_{x_i}(u_n-u)|^{p_i^-}_{p_i(x)} \right\rbrace \\ & + \min\left\lbrace |u_n-u|^{p_M^+}_{p_M(x)},|u_n-u|^{p_M^-}_{p_M(x)} \right\rbrace\bigg).
%\end{split}
%\end{equation}Hence,  combining \eqref{asump SS} and \eqref{long S}, we obtain the strong convergence $u_n \to u$ in $ W_0^{1, \pp(x)}(\Omega)$. Moreover, observe that \eqref{long S} also implies that $F$ is strictly monotone.

Regarding coerciveness, let $(u_n)_{n\in \mathbb{N}}$ be  a sequence in $W_0^{1, \pp(x)}(\Omega)$ such that $|u_n|_{1, \pp(x)}\to \infty$. Without loss of generality, assume that
$$|u_n|_{p_M(x)}, |\partial_{x_1}u_n|_{p_1(x)} \to \infty \quad \text{ and }\quad |\partial_{x_i}u_n|_{1, \pp(x)}\leq C$$for $i=2, ..., n$. If other terms go to infinite and other remain bounded, the reasoning is exactly the same.  Then
\begin{equation}\label{E11}
\begin{split}
|u_n|_{1, \pp(x)} & \leq \max\left\lbrace  \rho_{p_M}(u_n)^{1/p_M^+}, \rho_{p_M}(u_n)^{1/p_M^-}\right\rbrace + \sum_{i=1}^n\max\left\lbrace  \rho_{p_i}(\partial_{x_i} u_n)^{1/p_i^+}, \rho_{p_i}(\partial_{x_i} u_n)^{1/p_i^-}\right\rbrace\\ & \leq C + \left(\rho_{p_M}(u_n)^{1/p_M^-} + \rho_{p_1}(\partial_{x_1}u_n)^{1/p_1^-}\right)\\ &  \leq C+\left(\rho_{p_M}(u_n) + \rho_{p_1}(\partial_{x_1}u_n)\right)^{1/p_m},
\end{split}
\end{equation}where
$$p_m = \min\left\lbrace p_M^{-}, p_1^{-}, ..., p_n^{-}\right\rbrace > 1.$$Hence,
$$\dfrac{\left\langle Fu_n, u_n \right\rangle}{|u_n|_{1, \pp(x)}}\geq \dfrac{\rho_{p_M}(u_n) + \rho_{p_1}(\partial_{x_1} u_n)}{C+\left(\rho_{p_M}(u_n) + \rho_{p_1}(\partial_{x_1}u_n)\right)^{1/p_m}} \to \infty,$$as $n\to \infty$. 
\end{proof}

\begin{proof}[Proof of Theorem \ref{existence}]
Recalling the operators $F$ and $S$ from Lemma \ref{lemma f} and Lemma \ref{lemma S}, we have that solving problem \eqref{main problems} is equivalent to solve
$$Fu = -Su, \quad u \in  W_0^{1, \pp(x)}(\Omega).$$Since $F$ is coercive,  hemicontinuous  and strictly monotone, we have by \cite[Theorem 26.A]{Z} that the inverse operator $T=F^{-1}$ exists, it is strictly monotone, demicontinuous and bounded. However, since $F$ may be not uniformly monotone, we cannot infer directly from \cite[Theorem 26.A]{Z} that $T$ is continuous. So, we proceed to prove it. Take a sequence $u_n \in W_0^{-1, \pp'(x)}(\Omega)$ such that $u_n \to u$ to some $u\in  W_0^{-1, \pp'(x)}(\Omega)$. Since $T$ is demicontinuous,
$$T(u_n) \rightharpoonup T(u).$$Moreover, observe that
$$\limsup_{n\to \infty}\left\langle F(T(u_n))-F(T(u)), T(u_n)-T(u) \right\rangle = \limsup_{n\to \infty}\left\langle F(T(u_n)), T(u_n)-T(u) \right\rangle+ o_n(1),$$and also
\begin{equation*}
\begin{split}
\left\langle F(T(u_n))-F(T(u)), T(u_n)-T(u) \right\rangle & \leq |u_n-u|_{-1, \pp'(x)}| T(u_n)-T(u)|_{1, \pp(x)}\\ & \leq C|u_n-u|_{-1, \pp'(x)} = o_n(1).
\end{split}
\end{equation*}Thus, since $F$ is a $(S_{+})$-operator, we conclude that $T(u_n)\to T(u)$ strongly in $W_0^{1, \pp(x)}(\Omega)$. This proves that $T$ is continuous. Moreover, by \cite[Lemma 5.2]{Ber}, $T$ is an $(S_+)$-operator. 

Recall that the operator $S$ is bounded, continuous and quasimonotone (since it is compact). Also, to solve problem \eqref{main problems} is equivalent to solve the following Hammerstein equation
\begin{equation}\label{H eq}
v+ S\circ T v=0, \quad v= F(u).
\end{equation}In order to solve \eqref{H eq} we will apply the Degree Theory described in Section \ref{degree theory}. We will prove first that the set
$$B=\left\lbrace v \in W^{-1, \pp'(x)}(\Omega): v+tS\circ T (v)=0, \text{ for some }t \in [0, 1] \right\rbrace$$is bounded. Let $v\in B$, then $ v+tS\circ T (v)=0$ for some $t \in [0, 1]$. Letting $u= T(v)$, we  compute
\begin{equation}\label{est boundd B}
\begin{split}
 \rho_{p_M}(u)+\sum_{i=1}^n\rho_{p_i}(\partial_{x_i} u)& =\left\langle Fu, u \right\rangle  \\ & = \left\langle v, Tv \right\rangle  \\ & = -t\left\langle S\circ T(v), Tv \right\rangle  \leq \int_{\Omega}|f(x, u, D u)||u|\,dx.
\end{split}
\end{equation}Assume, to get a contradiction that the set $T B$ is not bounded. Then, there is a sequence $v_n\in B$ such that
$$|u_n|_{1, \pp(x)}\to \infty, \quad u_n=Tv_n.$$Assume first that 
$$|u_n|_{p_M}\to \infty$$
and $$|\partial_{x_i}u_n|_{1, p_i(x)} \to \infty,$$for all $i$. This implies, in particular, by Proposition \ref{properties modulo} that
\begin{equation}
\rho_{p_M}(u_n), \rho_{p_i}(\partial_{x_i}u_n)\to \infty.
\end{equation}

Having in mind \eqref{est boundd B} and the growth assumptions (f2), we get the following estimates:

\begin{equation}\label{E1}
\begin{split}
\int_\Omega |k||u_n|\,dx & \leq C|k|_{q'(x)}|u_n|_{q(x)}\\& \leq C|k|_{q'(x)}|u_n|_{1, \pp(x)} \\ & \leq C\left(\max\left\lbrace  \rho_{p_M}(u_n)^{1/p_M^+}, \rho_{p_M}(u_n)^{1/p_M^-}\right\rbrace + \sum_{i=1}^n\max\left\lbrace  \rho_{p_i}(u_n)^{1/p_i^+}, \rho_{p_i}(u_n)^{1/p_i^-}\right\rbrace\right)\\ & \leq C\left(\rho_{p_M}(u_n)^{1/p_M^-} + \sum_{i=1}^n \rho_{p_i}(u_n)^{1/p_i^-}\right)\\ & \leq C\left(\rho_{p_M}(u_n)^{1/p_m} + \sum_{i=1}^n \rho_{p_i}(u_n)^{1/p_m}\right) \\ & \leq C\left(\rho_{p_M}(u_n) + \sum_{i=1}^n \rho_{p_i}(u_n)\right)^{1/p_m}.
\end{split}
\end{equation}
Next, we deal with the term in the growth bound (f2) depending on $u_n$:
\begin{equation}
\begin{split}
\int_{\Omega}|u_n|^{q(x)}\,dx &\leq \max\left\lbrace |u_n|_{q(x)}^{q^{-}}, |u_n|_{q(x)}^{q^{+}}\right\rbrace \\&\leq  C\max\left\lbrace |u_n|_{1, \pp(x)}^{q^{-}}, |u_n|_{1, \pp(x)}^{q^{+}}\right\rbrace\\ & =  C|u_n|_{1, \pp(x)}^{q^{+}}. 
\end{split}
\end{equation}Analogously to \eqref{E1}, we derive
\begin{equation}\label{E2}
\int_{\Omega}|u_n|^{q(x)}\,dx \leq C\left(\rho_{p_M}(u_n) + \sum_{i=1}^n \rho_{p_i}(u_n)\right)^{q^{+}/p_m}
\end{equation}Recall that $q^{+}<p_m$. Finally, we consider the next term with derivatives
\begin{equation}\label{E3}
\begin{split}
\int_{\Omega}|\partial_{x_i}u_n|^{r_i(x)}|u_n|\,dx & \leq C||\partial_{x_i}u_n|^{r_i(x)}|_{q'(x)}|u_n|_{q(x)} \\ & = C||\partial_{x_i}u_n|^{r_i(x)}|_{q'(x)r_i(x)/r_i(x)}|u_n|_{q(x)}\\& \leq C\max\left\lbrace |\partial_{x_i}u_n|^{r_i^{+}}_{q'(x)r_i(x)},|\partial_{x_i}u_n|^{r_i^{-}}_{q'(x)r_i(x)} \right\rbrace |u_n|_{1, \pp(x)} \\ &\leq C|\partial_{x_i}u_n|_{p_i(x)}^{r_i^+}|u_n|_{1, \pp(x)}\\ & \leq C \rho_{p_i}(\partial_{x_i}u_n)^{r_i^+/p_i^-}|u_n|_{1, \pp(x)}.
\end{split}
\end{equation}Again, by the calculation from \eqref{E1}, we obtain
\begin{equation}\label{E4}
\int_{\Omega}|\partial_{x_i}u_n|^{r_i(x)}|u_n|\,dx \leq C\left(\rho_{p_M}(u_n) + \sum_{j=1}^n \rho_{p_j}(\partial_{x_j}u_n)\right)^{r_i^{+}/p_i^-+1/p_m}
\end{equation}Recall assumption (f2) that implies$$\dfrac{r_i^+}{p_i^-} +\dfrac{1}{p_m} < 1.$$
 Combining \eqref{E1}, \eqref{E2}, and \eqref{E4} with the growth assumption (f2) and \eqref{est boundd B}, we obtain a contradiction.  If not all the terms in the norms $|u_n|_{1, \pp(x)}$ go to infinite, then similar bounds to \eqref{E1}, \eqref{E2}, and \eqref{E4} hold, except that a constant should be added. However, the contradiction is still obtained.   Hence, $TB$ is a bounded set. Consequently, by the definition of $B$ and since $S$ is bounded, we obtain that $B$ itself is bounded.

We conclude with a usual argument in the application of Degree Theory. Choose $R> 0$ such that $|v|< R$ for all $v\in B$. Hence, for $v \in \partial B_R(0)$, we get
$$v+tS\circ T v \neq 0, \quad \text{for all }t\in [0, 1].$$From Lemma \ref{lemma KH}, 
$$I+S\circ T\in \mathcal{F}_T(\overline{B_R(0)}) \quad \text{and }\quad I= F\circ T\in \mathcal{F}_T(\overline{B_R(0)}).$$Consider next the homotopy $H:[0,1]\times \overline{B_R(0)}\to X$ given by
$$H(t, v):= v + t S\circ T v,\quad (t, v)\in  [0,1]\times \overline{B_R(0)}.$$
Applying the homotopy invariance and normalization properties of the degree (see Theorem \ref{main theorem degree}), we obtain that
$$d(I+S\circ T, B_R(0), 0)= d(I, B_0(0), 0)= 1.$$Hence, by the existence property in Theorem \ref{main theorem degree}, there is some $v\in B_R(0)$ such that
$$v+S\circ T v =0.$$Thus, $u=Tv$ is a solution of \eqref{main problems}. As we want to prove.

\end{proof}
\section*{Acknowledgments}
This work was partially supported by CONICET PIP 11220210100238CO and
ANPCyT PICT 2019-03837. P. Ochoa has been partially supported by Grant B017-UNCUYO.  P. Ochoa and A. Silva are members of CONICET.

\end{document}